\documentclass[pdflatex,sn-mathphys]{sn-jnl}
\usepackage[all,cmtip]{xy}

\jyear{2022}%

\theoremstyle{plain} 
\newtheorem{thm}{Theorem}[section] 
\newtheorem{cor}[thm]{Corollary} 
\newtheorem{lem}[thm]{Lemma} 
\newtheorem{prop}[thm]{Proposition}

\theoremstyle{definition} 
\newtheorem{defn}{Definition}[section]

\theoremstyle{remark} 
\newtheorem{rem}{Remark}

\begin{document}

\title[]{An equivalent definition for multiplicities of a quaternionic eigenvalue}


\author*[1,2]{\pfx{Dr} \fnm{Stefano} \sur{Spessato} \sfx{ORCID: 0000-0003-0069-5853}\email{stefano.spessato@unicusano.it }}

\affil*[1]{\orgname{Università degli Studi "Niccolò Cusano"}, \orgaddress{\street{via don Carlo Gnocchi 3}, \city{Rome}, \postcode{00166}, \country{Italy}}}


\abstract{In this paper we introduce two definitions for algebraic and geometric multiplicities of a quaternion right eigenvalue. This definitions are equivalent to the classical ones. However, differently from the usual definitions, the notions of \textit{root subspace} and Jordan form of a quaternion matrix are not required and all the properties can be proved easily.}




\maketitle
\subsection*{Acknowledgments}
I would like to thank the Unione Matematici Italiani (U.M.I.) for giving to me the opportunity of held a talk during the First U.M.I. meeting of Ph.D. students in May. Moreover I would also thank Alfredo Donno and Matteo Cavaleri for their advice and their help.
\\
\\Data sharing not applicable to this article as no datasets were generated or analysed during the current study.

\section*{Introduction}
Spectral properties of matrices with quaternion entries are an important research topic in last years. This happens for the large number of applications of quaternions both in Pure and Applied Mathematics. 
\\However, since the space of quaternions $\mathbb{H}$ is not a commutative algebra, is quite difficult to obtain the quaternion version of the classical theorems of Complex Linear Algebra. One of this result is the Quaternion Spectral Theorem  of Farenick and Pidkowich \cite{Spec}. Sometimes it is also difficult to define the equivalent notions: for example a quaternion matrix has at least two kinds spectra, which are called the \emph{right} spectrum and the \emph{left} spectrum. In particular, in this case, the most common choice is to study the right spectrum of a quaternion matrix.
\\In this paper, after some basiscs of Quaternion Linear Algebra (Section \ref{basi}), we give two new definitions for the algebraic and geometric multiplicity (resp. \emph{a.m.} and \emph{g.m.}) of a right eigenvalue of a quaternion matrix (Section \ref{2}). Moreover, always in the same Section, we also give the proofs of the main properties of this multiplicities. In particular we prove
\begin{itemize}
    \item the invariance of the multiplicities under the  similarity relation
    \item the classical inequality $ a.m.(q,A) \geq g.m.(q,A)$, where $q$ is a right-eigenvalue of a quaternion matrix $A$,
    \item the equality $ a.m.(q,A) = g.m.(q,A)$, if $A$ is a normal matrix.
\end{itemize}
Finally, in Section \ref{3}, we prove that our definitions are equivalent to the classical ones introduced in the book of Rodman \cite{Quat}.
\\The advantage of our definitions is that they can be given without the notions of \emph{root subspace} or \emph{Jordan form} of a quaternion matrix and so they can be used also by people which are not in Quaternion Linear Algebra. Moreover also the main properties can be proved without this tools. 
\section{Basics of Quaternion Linear Algebra}\label{basi}
\subsection{Quaternions}
Let $\mathbb{H}$ be the unital associative $\mathbb{R}$-algebra whose generators are $i$, $j$ and $k$ where
\begin{equation}
    i^2 = j^2 = k^2 = ijk = -1.
\end{equation}
This algebra is called \textbf{the algebra of real quaternions}. Given $q = a + bi + cj + dk$ in $\mathbb{H}$ the \textit{imaginary part} of $q$ is $Im(q) = bi + cj + dk$ and the \textit{real part} of $q$ is $Re(q) = a$.
Then \textit{conjugate} of a quaternion $q$ is the quaternion $\overline{q} := Re(q) - Im(q)$. 
\\$\mathbb{H}$ is not a commutative algebra, indeed $ij = k$, but
\begin{equation}
ji = j i (jj)(-1) = j(ij)j(-1) = jkj(-1) = ij(-1) = -k. 
\end{equation}
Two quaternions $q_1, q_2 \in \mathbb{H}$ are \emph{similar} if there is $h$ in $\mathbb{H}$ such that
\begin{equation}
    q_1 = h^{-1} q_2 h.
\end{equation}
If $q_1$ and $q_2$ are similar, then we write $q_1 \sim q_2$ and we denote by $[q_1]$ the class of $q$ in $\mathbb{H}$ such that $q \sim q_1$. Observe that if $Im(q) = 0$, then $[q] = \{q\}$, indeed $qh = hq$ for each $h$ in $\mathbb{H}$. By Lemma 2.2 of the paper of Farenick and Pidkowich \cite{Spec} we know that for each $q$ in $\mathbb{H}$ there is a unique $\lambda_q$ in $\mathbb{C}$ such that $Im(\lambda_q) \geq 0$ and
\begin{equation}
[q] = [\lambda_q] = [\overline{\lambda_q}]. 
\end{equation}
\subsection{Right spectrum of a quaternion matrix}
Let us denote by $\mathbb{H}^n$ the space of column vectors of length $n$ with quaternionic entries. Let $\mathbb{H}^{m \times n}$ be the space of $m \times n$ matrices with quaternionic entries. Given a matrix $A =(a_{ij})$ in $\mathbb{H}^{m \times n}$ we define the \textit{conjugate transpose} of $A$ as the $n \times m$ matrix $A^* := (\overline{a_{ji}})$. A square matrix $A$ is \textit{hermitian} if $A= A^*$ and it is normal if $AA^* = A^*A$.
\\For a square matrix $A$ in $\mathbb{H}^{n \times n}$, we say that $q \in \mathbb{H}$ is a \textit{right eigenvalue} if there is a $x \neq 0$ in $\mathbb{H}^n$ such that $Ax = x\cdot q$.
We say \textit{right spectrum} of $A$ the set $\sigma_r(A)$ of all right eigenvalues of $A$.
\\Notice that if $q$ is a right eigenvalue of $A$ and $q'$ is an element of $\mathbb{H}$ such that $q \sim q'$ then also $q'$ is a right eigenvalue of $A$. Indeed if $q' = sqs^{-1}$ and $Aw = w\cdot q$ then $A(ws^{-1}) = (ws^{-1})q'$.
\\Observe that each quaternion $q = a + bi + cj + dk$ can be seen as $q = (a + ib) + (c+id)j$ where $a + i b$ and $c + id$ are complex numbers. In the same way every quaternion matrix $A$ in $\mathbb{H}^{n \times n}$ can be decomposed as $A = A_1 + A_2j$ where $A_1$ and $A_2$ are $n \times n$-complex matrices.
\\Then the \textit{complex adjoint matrix} of $A$ is the matrix $f(A)$ in $\mathbb{C}^{2n \times 2n}$ defined as
\begin{equation}
    f(A) := \begin{bmatrix}
    A_1 && A_2 \\ -\overline{A}_2 && \overline{A}_1
    \end{bmatrix}
\end{equation}
\begin{prop}[pag. 78 of \cite{Spec}]
For each $A,B$ in $\mathbb{H}^{n\times n}$ the following properties hold:
\begin{itemize}
\item $f(Id_n) = Id_{2n}$;
    \item $f(A + B) = f(A) + f(B)$;
    \item $f(A\cdot B) = f(A) \cdot f(B);$
    \item $f(A^*) = f(A)^*$;
    \item $f(A^{-1}) = f(A)^{-1}$.
\end{itemize}
\end{prop}
These properties are very useful in order to investigate the right spectrum of a quaternion matrix. In particular we obtain the following result
\begin{prop}\label{f=r}
Let $A$ be a matrix in $\mathbb{H}^{n\times n}$. Then 
\begin{equation}
    Spec(f(A)) = \{\lambda \in \mathbb{C} \vert [\lambda] = [q] \mbox{   for each   } q \in \sigma_r(A)\}.
\end{equation}
In particular, if $A = A^*$, then $\sigma_r(A) \subset \mathbb{R}$ and
\begin{equation}
    Spec(f(A)) = \sigma_r(A).
\end{equation}
\end{prop}
\begin{proof}
This is essentially the proof of points 1. and 2. of Lemma 2.5 of \cite{Spec}.
\end{proof}
Observe that if $A$ is a $1 \times 1$-matrix, then $f: \mathbb{H} \longrightarrow \mathbb{C}^{2\times 2}$. So, for each matrix $A = (a_{rs})$ in $\mathbb{H}^{n\times n}$, we obtain a block matrix $\pi(A) := (\pi(a_{rs})) = (f(a_{rs}))$ in $\mathbb{C}^{2n \times 2n}$.
\begin{prop}\label{f=pi}
For each matrix $A$ in $\mathbb{H}^{n\times n}$, the matrices $f(A)$ and $\pi(A)$ are similar. This implies that
\begin{equation}
    Spec(f(A)) = Spec(\pi(A))
\end{equation}
and, if we denote by $a.m.(\lambda, B)$ and $g.m.(\lambda, B)$ the algebraic multiplicity and the geometric multiplicity of an eigenvalue $\lambda$ of a matrix $B$, then
\begin{equation}
a.m.(\lambda, \pi(A)) = a.m.(\lambda, f(A)) \mbox{     and    } g.m.(\lambda, \pi(A)) = g.m.(\lambda, f(A)).
\end{equation}
\end{prop}
\begin{proof}
The proof of this fact is similar to the the proof Theorem 5.9 of the work of Donno, Cavaleri and S. \cite{Godsil}. In their work, indeed, it is proved that, given a quaternion matrix $A$, then $\pi(A)$ and $f(A)$ are similar as complex matrices. So this fact implies also our statement. However, for sake of completeness, we report the proof of the similarity of $f(A)$ and $\pi(A)$.
\\Notice that, for each entry $a_{rs}$, of $A$ there are two complex numbers $a_{rs}^{[2]}$ and $a_{rs}^{[1]}$ such that
\begin{equation}
    a_{rs} = a_{rs}^{[1]} + a_{rs}^{[2]}j.
\end{equation}
Then, the matrix $\pi(A)$, is given by 
\begin{equation}
\pi(A) = \begin{bmatrix} a_{1,1}^{[1]} & a_{1,1}^{[2]} & ...  & a_{1,n}^{[1]} & a_{1,n}^{[2]} \\
-\overline{a}_{1,1}^{[2]} & \overline{a}_{1,1}^{[1]}  & ... & -\overline{a}_{n,n}^{[2]} & \overline{a}_{n,n}^{[1]}\\
\vdots & \vdots&  ... & \vdots& \vdots\\
a_{n,1}^{[1]} & a_{n,1}^{[2]} & ...  & a_{n,n}^{[1]} & a_{n,n}^{[2]} \\
-\overline{a}_{n,1}^{[2]} & \overline{a}_{n,1}^{[1]} & ... & -\overline{a}_{n,n}^{[2]} & \overline{a}_{n,n}^{[1]}\\
\end{bmatrix}
\end{equation}
Let us denote by $\tau:\{1, ..., 2n\} \longrightarrow \{1, ..., 2n\}$ the permutation defined as
\begin{equation}
    \tau(r) = \begin{cases} n + k \mbox{ if   } r=2k \\
    k+1 \mbox{ if   } r=2k +1
    \end{cases}
\end{equation}
Let $e_b$ be the $2n$-vector which has $1$ in the $b$-th entry and $0$ otherwise. Let us denote by $C_\tau$ the permutation $2n \times 2n$-matrix such that the $b$-th column is $e_{\tau(b)}$ and let $R_\tau$ be its transpose. Then
\begin{equation}
f(A) = R_\tau \pi(A) C_\tau.    
\end{equation}
Indeed
\begin{equation}
    \pi(A) C_\tau = \begin{bmatrix}  a_{1,1}^{[1]} & a_{1,2}^{[1]} & ... & a_{1, n}^{[1]} & a_{1,1}^{[2]} & ... & a_{1,n-1}^{[2]} & a_{1,n}^{[2]}\\
    -\overline{a}_{1,1}^{[2]} & -\overline{a}_{1,2}^{[2]} & ... & -\overline{a}_{1, n}^{[2]} & \overline{a}_{1,1}^{[1]} & ... & \overline{a}_{1,n-1}^{[1]} & \overline{a}_{1,n}^{[1]}\\
    \vdots & \vdots & & \vdots & \vdots & & \vdots & \vdots \\
    a_{n,1}^{[1]} & a_{n,2}^{[1]} & ... & a_{n, n}^{[1]} & a_{n,1}^{[2]} & ... & a_{n,n-1}^{[2]} & a_{n,n}^{[2]}\\
    -\overline{a}_{n,1}^{[2]} & -\overline{a}_{n,2}^{[2]} & ... & -\overline{a}_{n, n}^{[2]} & \overline{a}_{n,1}^{[1]} & ... & \overline{a}_{n,n-1}^{[1]} & \overline{a}_{n,n}^{[1]}
    \end{bmatrix}
\end{equation}
and
\begin{equation}
    R_\tau \pi(A) C_\tau = \begin{bmatrix}  a_{1,1}^{[1]} & a_{1,2}^{[1]} & ... & a_{1, n}^{[1]} & a_{1,1}^{[2]} & ... & a_{1,n-1}^{[2]} & a_{1,n}^{[2]}\\
    a_{2,1}^{[1]} & a_{2,2}^{[1]} & ... & a_{2, n}^{[1]} & a_{2,1}^{[2]} & ... & a_{2,n-1}^{[2]} & a_{2,n}^{[2]}\\
    \vdots & \vdots & &\vdots &\vdots & &\vdots & \vdots \\
    a_{n,1}^{[1]} & a_{n,2}^{[1]} & ... & a_{n, n}^{[1]} & a_{n,1}^{[2]} & ... & a_{n,n-1}^{[2]} & a_{n,n}^{[2]}\\
    -\overline{a}_{1,1}^{[2]} & -\overline{a}_{1,2}^{[2]} & ... & -\overline{a}_{1, n}^{[2]} & \overline{a}_{1,1}^{[1]} & ... & \overline{a}_{1,n-1}^{[1]} & \overline{a}_{1,n}^{[1]}\\
    \vdots & \vdots & &\vdots &\vdots & &\vdots & \vdots \\
   -\overline{a}_{n-1,1}^{[2]} & -\overline{a}_{n-1,2}^{[2]} & ... & -\overline{a}_{n-1, n}^{[2]} & \overline{a}_{n-1,1}^{[1]} & ... & \overline{a}_{n-1,n-1}^{[1]} & \overline{a}_{n-1,n}^{[1]}\\
    -\overline{a}_{n,1}^{[2]} & -\overline{a}_{n,2}^{[2]} & ... & -\overline{a}_{n, n}^{[2]} & \overline{a}_{n,1}^{[1]} & ... & \overline{a}_{n,n-1}^{[1]} & \overline{a}_{n,n}^{[1]}
    \end{bmatrix}
\end{equation}
Then if $A = (a_{rs})$, then
\begin{equation}
 R_\tau \pi(A) C_\tau = 
    \begin{bmatrix}
    a_{rs}^{[1]} & a_{rs}^{[2]} \\
    -\overline{a}_{rs}^{[2]} & \overline{a}_{rs}^{[1]}
    \end{bmatrix} = f(A).
\end{equation}
Notice that, since $R_\tau$ is a permutation matrix, then $R_\tau^{-1} = C_\tau$ and so $\pi(A)$ and $f(A)$ are similar. This means that they have the same spectrum and the same multiplicities. 
\end{proof}
\subsection{Schur’s triangularization theorem}
We conclude this Section with a classical result on Quaternion Linear Algebra. This is an immediate consequences of Fundamental Theorem of Algebra. A short proof of this result can be found in the book of Rotman\cite{Quat}
\begin{lem}[Theorem 5.3.5. of \cite{Quat}]
For every $A \in \mathbb{H}^{n\times n}$ there exists an eigenvalue which is a complex number with nonnegative imaginary part.
\end{lem}
As a consequence of this result, by applying the usual technique from Complex Linear Algebra, it is possible to prove a quaternionic version of Schur’s triangularization theorem.
\begin{lem}[Lemma 3.3 of \cite{Spec}, Theorem 5.3.6 of \cite{Quat}]\label{lemmA}
Let $A$ be a matrix in $\mathbb{H}^{n \times n}$. Then there are $U, T$ in $\mathbb{H}^{n \times n}$ such that $U$ is unitary, $T$ is an upper triangular matrix and 
\begin{equation}
T = U^*AU.
\end{equation}
Moreover the entries on the diagonal of $T$ are complex numbers $\mu_i$ such that $Im(\mu_i) \geq 0.$
\end{lem}
\subsection{Quaternionic dimension of a right $\mathbb{H}$-vector space}
Our next step is to define the geometric multiplicity of a right eigenvalue $q$ of a quaternionic matrix.
\begin{defn}
An additive abelian group $V$ is a right $\mathbb{H}$-vector space if there is a map $V \times \mathbb{H} \longrightarrow V$ under which the image of each pair $(v, q)$ is denoted by $vq$, such that
for all $q, q_1, q_2$ in $\mathbb{H}$ and $v, v_1, v_2$ in $V$:
\begin{itemize}
    \item $(v_1 + v_2)q = v_1q + v_2q$,
    \item $v(q_1 + q_2) = v q_1 + v q_2$,
    \item $v(q_1q_2) = (v q_1)q_2$, and
    \item $v1 = v$.
\end{itemize}
\end{defn}
\begin{defn}
If $V$ is a nonzero right $\mathbb{H}$-vector space and if $X \subseteq V$ is a nonempty subset, then:
\begin{enumerate}
    \item $X$ is a \textbf{generating set} for $V$ if for each $v$ in $V$, 
    \begin{equation}
       v = v_1 q_1 + ... + v_kq_k,
    \end{equation}
    where $v_i \in X$ and $q_i \in \mathbb{H}$  for each $i = 1, ..., k$,
    \item $X$ is an $\mathbb{H}$\textbf{-independent set} if, for any $k \in \mathbb{N}_{>0}$ and for any distinct $v_1, ..., v_k$ in $V$ the equation 
    \begin{equation}
        v_1 q_1 + ... + v_kq_k = 0
    \end{equation}
    is satisfied only for $q_1 = ...= q_k = 0$.
    \item $X$ is a \textbf{basis} for $V$ if $X$ is generating and $\mathbb{H}$-independent set.
\end{enumerate}
\end{defn}
\begin{rem}
Every generating space $X$ of a right $\mathbb{H}$-vector space $V$ contains a basis of $V$: the proof is exactly the same as for real vector spaces.
\end{rem}
\begin{thm}[Theorem 4.2 of \cite{Spec}]
If a right  $\mathbb{H}$-vector space $V$ has a generating set with $m$ elements, then $V$ has a basis with $n$ elements for some $n \leq m$. Moreover:
\begin{enumerate}
\item every basis of $V$ has $n$ elements;
\item if $1 \leq k \leq n$ and if $v_1, ..., v_k $ in $V$ are $\mathbb{H}$-independent, then there is a basis of $V$ that contains the vectors $v_1, ..., v_k $.
\end{enumerate}
\end{thm}
\begin{defn}
Let $V$ be a right $\mathbb{H}$-vector space and let $X$ be a basis for $V$. Then the \textbf{quaternionic dimension} of $V$ is the cardinality of $X$. We denote it by $dim_{\mathbb{H}}(V)$.
\end{defn}
\begin{rem}
Let $A$ be a matrix in $\mathbb{H}^{n \times n}$. Let $q$ be a right eigenvalue of $A$. Let us denote by $V_q$ the right $q$-eigenspace, i.e. the subset of $\mathbb{H}^{n}$ given by the right $q$-eigenvectors. Notice that $V_q$ is a real vector space. Unfortunately it is not a right $\mathbb{H}$-vector space, indeed if $v$ is a $q$-eigenvector and $q_2$ is an element of $\mathbb{H}$, then $vq_2$ is a right $q_2^{-1} q q_2$-eigenvalue.
\end{rem}
\section{Multiplicities of a right eigenvalue}\label{2}
\subsection{Algebraic multiplicity}
Our next step is to define the algebraic multiplicity for a right eigenvalue of a quaternion matrix.
\begin{defn}\label{m.alg}
Let $A$ be a matrix in $\mathbb{H}^{n \times n}$ and let $q$ be a right eigenvalue of $A$. Let $\lambda_q$ be the complex number such that $Im(\lambda_q) \geq 0$ and $[q] = [\lambda_q]$. Let $U$ and $T = U^*AU$ be two matrices satisfying Lemma \ref{lemmA}. Then the \textbf{algebraic multiplicity } of $q$ (or $a.m.(q, A)$) is the number of occurrences of $\lambda_q$ in the diagonal of $T$.
\end{defn}
\begin{prop}\label{algeq}
The Definition \ref{m.alg} is well-given and it does not depend on the choice of $T$. Moreover, if $q$ is a right eigenvalue of a quaternionic matrix $A$ such that $[q] = [\lambda_q]$ where $\lambda_q \in \mathbb{C}$ and $Im(\lambda_q) \geq 0$, then 
\begin{equation}\label{uno}
a.m.(q, A) = a.m.(\lambda_q, f(A)) = a.m.(\lambda_q, \pi(A))
\end{equation}
if $Im(\lambda_q) \neq 0$, and 
\begin{equation}\label{due}
a.m.(q, A) = \frac{1}{2}a.m.(\lambda_q, f(A)) = \frac{1}{2}a.m.(\lambda_q, \pi(A))
\end{equation}
if $\lambda_q$ is a real number.
\end{prop}
\begin{proof}
Observe that $T = U^* A U$ implies that $\pi(T)$ and $\pi(A)$ are similar, indeed $\pi(T) \sim f(T) = f(U)^*f(A)f(U) \sim f(A) \sim \pi(A)$. So, if $[q]$ is a right eigenvalue of $A$, then $\lambda_q$ is an eigenvalue for $\pi(A)$ and also for $\pi(T)$. This means that $\lambda_q$ occurs in the diagonal of $\pi(T)$. Indeed $\pi(T)$ is a triangular matrix since the $i$-th diagonal block of $\pi(T)$ is $\begin{bmatrix}\mu_i & 0 \\ 0 & \overline{\mu_i}\end{bmatrix}$ where $\mu_i$ are the diagonal entries of $T$. So $\lambda_q$ occurs also in the diagonal of $T$ and, in particular, 
\begin{equation}
a.m.(q, A) = a.m.(\lambda_q, \pi(T)) = a.m.(\lambda_q, \pi(A)) = a.m.(\lambda_q, f(A))
\end{equation}
if $Im(\lambda) \neq 0$, and 
\begin{equation}
a.m.(q, A) = \frac{1}{2} a.m.(\lambda_q, \pi(T)) = \frac{1}{2}a.m.(\lambda, \pi(A)) = \frac{1}{2}a.m.(\lambda, f(A))
\end{equation}
otherwise.
\\Finally, if there are some matrices $U_1$ and $T_1 \neq T$ satisfying the Lemma \ref{lemmA}, then $T_1 = U_1^*U TU^*U_1$. So $\pi(T_1)$ and $\pi(T)$ are similar because $f(T_1)$ and $f(T)$ are similar. 
\\This proves that the definition of algebraic multiplicity is well-given.
\end{proof}
\subsection{Geometric multiplicity}
Our next step is to define the geometric multiplicity of a right eigenvalue $q$ of
a quaternionic matrix.
\begin{defn}
Let $A$ be a matrix in $\mathbb{H}^{n \times n}$ and let $q$ be a right eigenvalue of $A$. Then the $[q]$-\textbf{eigenspace} is the subset of $\mathbb{H}^n$
\begin{equation}
V_{[q]} := \sum\limits_{\tilde{q} \in [q]}V_{\tilde{q}} 
\end{equation}
which is the space of all finite sums of vectors $v_1 + ... + v_k$ where for each $i$ the vector $v_i$ is in $V_{q_i}$ for some $q_i$.
\end{defn}
\begin{rem}
Notice that $V_{[q]}$ is a right $\mathbb{H}$-vector space. Moreover, for each $\tilde{q}$ in $[q]$ is possible to find a basis of $V_{[q]}$ contained in $V_{\tilde{q}}$. Indeed if $v$ is in $V_{\tilde{q}}$ where $s\tilde{q}s^{-1} = q$ for some $s$ in $\mathbb{H}$, then $vs^{-1}$ is an element of $V_q$. Then, since $V_q$ is a real vector space, then we can find a basis\footnote{in the real vector spaces sense} $\{e_i\}$ of $V_q$ and some real coefficients $\{a^i\}$ such that $vs^{-1} = e_1a^1 + ... + e_ka^k$ and so
\begin{equation}
    v = e_1(a^1s) + ... + e_k(a^ks)
\end{equation}
This means that $\{e_i\}$ is a generating set for $V_{[q]}$ and so it contains a basis for the right $\mathbb{H}$-vector space $V_{[q]}$.
\end{rem}
\begin{defn}\label{m.geom}
Let $q$ be a right eigenvalue of $A$ in $\mathbb{H}^{n \times n}$. The \textbf{geometric multiplicity} of $q$ is
\begin{equation}
g.m.(q, A) := dim_{\mathbb{H}}V_{[q]}.
\end{equation}
\end{defn}
\begin{rem}
Let $A$ and $D$ be two quaternionic matrix. Assume the existence of an invertible matrix $U$ such that $D = U^{-1} A U$. Then the eigenvalues of $A$ and of $D$ are the same and they also have the same geometric multiplicity. Indeed, given an eigenvalue $q$ of $A$ and given $\{x_i\}$ a basis of the $[q]$-eigenspace, then $\{Ux_i\}$ is a basis for the $[q]$-eigenspace of $A$.
\end{rem}
\begin{prop}\label{geomeq}
Let $A$ be a matrix in $\mathbb{H}^{n \times n}$ and let $q$ be a right eigenvalue. Let $\lambda_q$ be the element of $\mathbb{C}$ such that $Im(\lambda_q) \geq 0 $ and $[q] = [\lambda_q]$. Then, if $Im(\lambda_q) \neq 0$
\begin{equation}
    g.m.(q, A) = g.m.(\lambda_q, f(A)) = g.m.(\lambda_q, \pi(A))
\end{equation}
and, if $\lambda_q$ is real,
\begin{equation}
    g.m.(q, A) = \frac{1}{2}g.m.(\lambda_q, f(A)) = \frac{1}{2}g.m.(\lambda_q, \pi(A)).
\end{equation}
\end{prop}
\begin{proof}
Fix a right eigenvalue $q$ of $A$. Then choose a basis $\{v_1, ..., v_d\}$ of $V_{[q]}$ which is contained in $V_{\lambda_q}$ and complete it in order to obtain a basis $X = \{v_1, ..., v_d, x_{d+1}, ..., x_n\} =: \{s_i\}$ of $\mathbb{H}^{n}$. Let $\mathcal{C}_{X}$ be the matrix which has in the $i$-th column the vector $s_i$ of $X$. Let $b_i$ be the vector in $\mathbb{H}^n$ which has the $i$-th entry equal to $1$ and $0$ otherwise. Observe that $\mathcal{C}_{X}$ is invertible and its inverse is given by the matrix which has in the $j$-th column the components of $b_i$ with respect to $X$. Then, we define
\begin{equation}
D_\lambda := \mathcal{C}_{X}^{-1} \cdot A \cdot \mathcal{C}_{X}. 
\end{equation}
The matrix $D_\lambda$ has the form
\begin{equation}
D_\lambda = \begin{bmatrix} \lambda_q Id_d & B_1 \\
0 & B_2
\end{bmatrix}
\end{equation}
The matrix $f(D_\lambda) = f(\mathcal{C}_{X}^{-1})f(A)f(\mathcal{C}_{X})$ is given by
\begin{equation}
f(D_\lambda) := \begin{bmatrix}\lambda_q Id_d & \star & 0 & \star \\
0 & \star & 0 &\star\\
0 & \star & \overline{\lambda}_q Id_d & \star\\
0 & \star & 0 &\star\end{bmatrix}.
\end{equation}
Since $f(A)$ is similar to $f(D_\lambda)$, we obtain that
\begin{equation}
g.m.(\lambda_q, f(D_\lambda)) = g.m.(\lambda_q, f(A)) \geq g.m.(q, A)   
\end{equation}
if $Im(\lambda_q) > 0$ and
\begin{equation}
g.m.(\lambda_q, f(D_\lambda)) = g.m.(\lambda_q, f(A)) \geq 2g.m.(q, A)   
\end{equation}
if $\lambda_q$ is real. Indeed if $\lambda_q$ has $Im(\lambda_q) >0 $ then the $\lambda_q$-eigenspace of $f(D_\lambda)$ contains $Span(e_1, ..., e_d)$ where $e_i$ is the vector which has $1$ in the $i$-the entry and zero otherwise.
On the other hand, if $\lambda_q$ is real, then the $\lambda_q$-eigenspace of $f(D_\lambda)$ contains $Span(e_1, ..., e_d, e_{n+1}, ..., e_{n+d})$.
\\For the other inequality, assume that there is a vector $w = (w_A, w_B)$ in $\mathbb{C}^{2n}$ such that $f(D_\lambda)w = w \lambda_q$. By Lemma 2.4 of \cite{Spec}, we know that $y := w_A - \overline{w}_B \cdot j$ in $\mathbb{H}^n$ is a vector in $V_{\lambda_q}$. This means that, with respect to the basis $X$,
\begin{equation}
\begin{split}
y &= v_1a^1 + ... + v_da^d  +  x_{d+1} \cdot 0 + ... + x_n\cdot 0\\
&= v_1(w_A^1 - \overline{w}_B^1j) + ... + v_d(w_A^d - \overline{w}_B^dj)\\
&+ x_{d+1} \cdot (w_A^{d+1} - \overline{w}_B^{d+1}j) + ... + x_n\cdot (w_A^{n} - \overline{w}_B^{n}j).
\end{split}
\end{equation}
So, for each $r = d+1, ..., n$, we obtain that $w_A^r = w_B^r = 0$. This means that $w$ is in $Span(e_1, ..., e_d, e_{n+1}, ..., e_{n+d})$. This is sufficient to conclude if $\lambda_q$ is real. If $\lambda_q$ is not real, observe that $Span(e_{n+1}, ..., e_{n+d})$ is contained in the $\overline{\lambda}$-eigenspace of $D_\lambda$ and so, in particular, $w$ is in $Span(e_1, ..., e_d)$.
\end{proof}
\subsection{Main properties of multiplicities}
The first property of algebraic and geometric multiplicities is the classical inequality between them.
\begin{prop}
Let $A \in \mathbb{H}^{n \times n}$. Then, if $q$ is a right eigenvalue of $A$, then
\begin{equation}
    a.m.(q) \geq g.m.(q).
\end{equation}
\end{prop}
The second one is about the invariance of multiplicities
\begin{prop}\label{inv}
Let $A$ and $B$ two matrices in $\mathbb{H}^{n \times n}$. Assume that they are similar. Then their right spectrum coincides and if $q$ is an eigenvalue of $A$ and $B$, then
\begin{equation}
    a.m.(q, A) = a.m.(q, B) \mbox{   and   } g.m.(q, A) = g.m.(q, B).
\end{equation}
\end{prop}
The following Corollary holds when $A$ is a normal matrix. In order to introduce it, we need a result from \cite{Spec}.
\begin{thm}[Theorem 3.3 of \cite{Spec}]\label{Spectral Theorem}
If $A \in \mathbb{H}^{n \times n}$ is normal, i.e. $AA^* = A^*A$ then there are matrices $D$, $U$ in $\mathbb{H}^{n \times n}$ such that 
\begin{itemize}
    \item $U$ is a unitary matrix, $D$ is a diagonal matrix, and $U^*AU = D$;
    \item each diagonal entry of $D$ is a complex number $\lambda$ such that $Im(\lambda) = 0$;
    \item $q \in \mathbb{H}^{n \times n}$ is a right eigenvalue of $A$ if and only if $[q] = [\lambda]$ for some diagonal element $\lambda$ of $D$.
\end{itemize}
\end{thm}
\begin{prop}
Let $A \in \mathbb{H}^{n \times n}$ be a normal matrix. Then, if $q$ is a right eigenvalue of $A$, then
\begin{equation}
    a.m.(q) = g.m.(q).
\end{equation}
\end{prop}
\section{Equivalence of the definitions}\label{3}
As stated in the Introduction, the definitions given Section \ref{2} are equivalent to the definitions of \cite{Quat}. In this final section we quickly introduce the classical definitions and we show this equivalence.
\\In order to define the classical algebraic and the geometric multiplicity of a right eigenvalue we need the notions of minimal polynomial $p^A$ related to $A$ and of the root subspace $\mathcal{M}$ of $A$ corresponding to $p^A$.
\\Notice that $\mathbb{H}^{n \times n}$ is a finite dimensional, real vector space. Then there are some polynomials $p(t) \in \mathbb{R}[t]$ such that $p(A) = 0$.
\begin{defn}
The \textbf{minimal polynomial of} $A$ is the polynomial $p^A(t) \neq 0$ which is monic and it has minimal degree.
\end{defn}
\begin{rem}
The minimal polynomial $p^A$ of a matrix $A \in \mathbb{H}^{n \times n}$ is unique.
\end{rem}
Assume that
\begin{equation}\label{polynomial}
p^A(t) = p_1(t)^{m_1} \cdot ... \cdot p_k(t)^{m_k}
\end{equation}
where $p_j(t)$ are distinct irreducible real polynomial which are monic. Notice that for each eigenvalue $\lambda$ of $A$ there is exactly one polynomial $p_j$ such that $\lambda$ is a root of $p_j$. Then the following definitions make sense.
\begin{defn}
A \textbf{root subspace} $\mathcal{M}_j$ of $A$ is the set
\begin{equation}
    \mathcal{M}_j := \{u \in \mathbb{H}^n \vert p_j(A)^{m_j}u = 0 \}
\end{equation}
where $j = 1, ..., k$. The \textbf{root subspace of the eingenvalue} $\lambda$ is the root subspace related to the polynomial $p_j$ which has $\lambda$ as a root.
\end{defn}
\begin{rem}
    Given a right eigenvalue $q$ of a quaternion matrix $A$, then $V_{[q]} \neq \mathcal{M}_j$, where $\mathcal{M}_q$ is the root subspace related to $q$. An example is the following: notice that $0$ is an eigenvalue of the matrix
    \begin{equation}
    A = \begin{bmatrix}
    0 & 0 & 0 & 0 \\
    0 & 0 & 0 & 0 \\
    1 & 0 & 0 & 0 \\
    0 & 1 & 0 & 0 \\
    \end{bmatrix}.    
    \end{equation}
Then the minimal polynomial of $A$ is $p^A(t) = t^2$ and so $\mathcal{M}_0 = ker(A^2) = \mathbb{H}^4$. So $dim_{\mathbb{H}}(\mathcal{M}_0) = 4$. On the other hand, we know that the quaternionic dimension of $V_{[0]}$ is equal to $\frac{1}{2}g.m.(0, f(A)) = 2$.
\end{rem}
\begin{defn}
Let $\lambda$ be an eigenvalue of $A$. Let $p_j(t)$ be the unique polynomial in Equality \ref{polynomial} such that $\lambda$ is a root of $p_j(t)$. Then the \textbf{quaternion geometric multiplicity} of $\lambda$ (Definition in pag. 99 of \cite{Quat}) is the number the maximal number of $\mathbb{H}$-linearly independent eigenvectors of $A$ in $\mathcal{M}_j$. We denote this number by $q.g.m.(\lambda)$.
\\On the other hand, the \textbf{quaternion algebraic multiplicity} of $\lambda$ (Definition in pag. 99 of \cite{Quat}) is defined as the quaternion dimension of $\mathcal{M}_\lambda$. We will denote this number by $q.a.m.(\lambda)$.
\end{defn}
\begin{rem}
In general $q.a.m.(\lambda) \geq q.g.m.(\lambda)$, indeed on $\mathcal{M}_\lambda$ there could be some vectors $v$ which are not eigenvectors of $A$.
\end{rem}
The equivalence of the definitions is a Corollary of the existence and the unicity \footnote{unless to some permutation of blocks} of a Jordan form for each quaternion matrix $A$. In particular, on the diagonal of the Jordan form of a quaternion matrix there are complex numbers $\lambda_i$ such that $Im(\lambda_i) \geq 0$. The proof of these facts is quite long and it is proved in Section 5.6. of \cite{Quat} (Theorem 5.5.3., from p.102 to p.109).
\\In particular we need Theorem 5.5.5. of \cite{Quat}.
\begin{thm}[Theorem 5.5.5., \cite{Quat}]
Let $A \in \mathbb{H}^{n \times n}$ and let $\lambda \in \sigma(A)$. Then:
\begin{itemize}
    \item the number of Jordan blocks in the Jordan form of $A$ corresponding to the
eigenvalues similar to $\lambda$ is equal to $q.g.m.(\lambda)$;
\item the sum of the sizes of Jordan blocks in the Jordan form of $A$ corresponding to the eigenvalues similar to $\lambda$ is equal to $q.a.m.(\lambda)$.
\end{itemize}
\end{thm}
\begin{cor}
Let $A$ be a matrix in $\mathbb{H}^{n \times n}$. Then, for each eigenvalue $\lambda$ of $A$,
\begin{equation}
    a.m.(q, A) = q.a.m.(q, A) \mbox{     and     } g.m.(q, A) = q.g.m.(q, A).
\end{equation}
\end{cor}
\begin{proof}
    Let $J$ be the Jordan form of $A$. Notice that $J$ and $A$ are similar. Let $q$ be a right eigenvalue of $A$ and let $\lambda_q$ be the complex number with $Im(\lambda_q) \geq 0$ and $\lambda_q \sim q$. Then, if $Im(\lambda_q) \neq 0$
\begin{equation}
    q.g.m.(q, A) = q.g.m.(\lambda_q, J) = g.m.(\lambda_q, f(J))
\end{equation}
and, if $\lambda_q$ is real,
\begin{equation}
    q.g.m.(q, A) = q.g.m.(\lambda_q, J) = \frac{1}{2}g.m.(\lambda_q, f(J)).
\end{equation}
Moreover
\begin{equation}
    q.a.m.(q, A) = q.a.m.(\lambda_q, J) = a.m.(\lambda_q, f(J))
\end{equation}
and, if $\lambda_q$ is real,
\begin{equation}
    q.a.m.(q, A) = q.a.m.(\lambda_q, J) = \frac{1}{2}a.m.(\lambda_q, f(J)).
\end{equation}
Since $A$ and $J$ are similar, then we conclude by appliying Propositions \ref{inv}, \ref{geomeq} and \ref{algeq}.
\end{proof}
	\phantomsection
	\bibliographystyle{sapthesis} 

\addcontentsline{toc}{section}{\bibname}
	\begin{thebibliography}{99}
	\bibitem{Godsil} A. Donno, M. Cavaleri, S. Spessato, \emph{Godsil-Mckay switching for gain graphs}, 2022, preprint on arxiv: https://arxiv.org/pdf/2207.10986.pdf
\bibitem{Spec} D. R. Farenick, B. A.F. Pidkowich, \emph{The spectral theorem in quaternions}, Linear Algebra
Appl., 371 (2003) 75–102, MR 1997364
\bibitem{Quat} L. Rodman, \emph{Topics in Quaternion Linear Algebra}. Princeton University Press, 2014. JSTOR, http://www.jstor.org/stable/j.ctt6wpz0p.
	\end{thebibliography}
	\end{document}